\documentclass[a4paper,10pt,reqno]{amsart}
\usepackage{amssymb,amsmath,mathrsfs,graphics,graphicx,mathptm,float}
\usepackage[T1]{fontenc}
\usepackage[english, french]{babel}
\usepackage[all]{xy}
\theoremstyle{plain}
\newtheorem{thm}{Theorem}[section]
\newtheorem{pro}[thm]{Proposition}
\newtheorem{lem}[thm]{Lemma}

\newtheorem{theoalph}{Theorem}

\newtheorem{proalph}[theoalph]{Proposition}
\newtheorem{factalph}[theoalph]{Fact}

\theoremstyle{definition}

\newtheorem{rem}[thm]{Remark}

\usepackage[colorlinks, linktocpage, citecolor = blue, linkcolor = blue]{hyperref}

\addtocounter{section}{0}             
\numberwithin{equation}{section}       

\begin{document}
\title[Degree growth of polynomial automorphisms and birational maps: new examples]{Degree 
growth of polynomial automorphisms \\
and birational maps: some examples}
\selectlanguage{english}

\author{Julie D\'eserti}
\address{Universit\'e Paris Diderot, Sorbonne Paris Cit\'e, Institut de 
Math\'ematiques de Jussieu-Paris Rive Gauche, UMR $7586$, CNRS, Sorbonne 
Universit\'es, UPMC Univ Paris $06$, F-$75013$ Paris, France.}
\email{deserti@math.univ-paris-diderot.fr}

\maketitle

\begin{abstract}
We provide the existence of new degree growths in the context of polynomial 
automorphisms of $\mathbb{C}^k$: if $k$ is an integer $\geq 3$, then for any 
$\ell\leq \left[\frac{k-1}{2}\right]$ there exist polynomial automorphisms $f$ 
of $\mathbb{C}^k$ such that $\deg f^n\sim n^\ell$. 
We also give counter-examples in dimension $k\geq 3$ to some 
classical properties satisfied by polynomial automorphisms of $\mathbb{C}^2$.

We provide the existence of new degree growths in the context of birational
maps of $\mathbb{P}^k_\mathbb{C}$: assume $k\geq 3$; forall $0\leq\ell\leq k$ 
there exist birational maps $\phi$ of $\mathbb{P}^k_\mathbb{C}$ such that 
$\deg \phi^n\sim n^\ell$.
\end{abstract}

\section{Introduction}\label{sec:intro}

Let $f$ be a polynomial automorphism of $\mathbb{C}^2$, then either 
$(\deg f^n)_{n\in\mathbb{N}}$ is bounded, or $(\deg f^n)_{n\in\mathbb{N}}$ grows 
exponentially. In higher dimensions there are intermediate growths:

\begin{theoalph}\label{thm:newdegree}
{\sl Let $k$ be an integer $\geq 3$. For any $\ell\leq \left[\frac{k-1}{2}\right]$ 
there exist polynomial automorphisms $f$ of $\mathbb{C}^k$ such that 
\[
\deg f^n\sim n^\ell.
\]
}
\end{theoalph}

\bigskip

The group of polynomial automorphisms of $\mathbb{C}^2$ has a structure of 
amalgamated product (\cite{Jung}); using this rigidity a lot of properties of 
polynomial automorphisms of $\mathbb{C}^2$ have been established. All these 
properties show a dichotomy; up to conjugacy there are two types of polynomial 
automorphisms of $\mathbb{C}^2$: the elementary ones and the H\'enon ones. 
Furthermore if $f$ and $g$ are two polynomial automorphisms of $\mathbb{C}^2$ 
and $H_\infty$ denotes the line at infinity (we view $\mathbb{C}^2$ in 
$\mathbb{P}^2_\mathbb{C}$), then 

\smallskip
\begin{itemize}
\item[$(\mathcal{P}_1)$]: $f$ is algebraically stable if and only if 
$(\deg f^n)_{n\in\mathbb{N}}$ grows exponentially;
\smallskip
\item[$(\mathcal{P}_2)$]: for any $n\geq 1$ the equality $\deg f^n=(\deg f)^n$ 
holds if and only if $f$ does not preserve a fibration in hyperplanes;
\smallskip
\item[$(\mathcal{P}_3)$]: the sequences  $(\deg f^n)_{n\in\mathbb{N}}$ and 
$(\deg g^n)_{n\in\mathbb{N}}$ have the same growth if and only if the configurations 
of $\big(f(H_\infty),f^{-1}(H_\infty),\mathrm{Ind}(f),\mathrm{Ind}(f^{-1})\big)$ 
and $\big(g(H_\infty),g^{-1}(H_\infty),\mathrm{Ind}(g),\mathrm{Ind}(g^{-1})\big)$ 
are the same;
\smallskip
\item[$(\mathcal{P}_4)$]: $\deg f^2=(\deg f)^2$ if and only if $\deg f^n=(\deg f)^n$ 
for any $n\in\mathbb{N}$ (\emph{see} \cite[Proposition 3]{Furter}).
\end{itemize}

\smallskip

Note that in $(\mathcal{P}_1)$, $(\mathcal{P}_2)$ and $(\mathcal{P}_4)$ one can 
add "if and only if $f$ is of H\'enon type".

\begin{factalph}\label{fact:counterexample}
{\sl We give counter-examples to Properties $(\mathcal{P}_1)$, 
$(\mathcal{P}_2)$, $(\mathcal{P}_3)$ and $(\mathcal{P}_4)$ in 
dimension $\geq 3$.}
\end{factalph}

Nevertheless one can prove a similar result to property $\mathcal{P}_4$:

\begin{proalph}
{\sl Let $f$ be a polynomial automorphism of $\mathbb{C}^k$. Then $\deg f^i=(\deg f)^i$ 
for $1\leq i\leq k$ if and only if $\deg f^n=(\deg f)^n$ for any $n\geq 1$.}
\end{proalph}

\bigskip

If $\phi$ is a birational self map of $\mathbb{P}^2_\mathbb{C}$, 
then $(\deg \phi^n)_{n\in\mathbb{N}}$ is either bounded, or grows linearly, or grows 
quadratically, or  grows exponentially (\cite{DillerFavre}). 
In this context there exist other types of growth in higher dimension. Lin has 
studied the degree growth of monomial maps of $\mathbb{P}^k_\mathbb{C}$ (\emph{see} 
\cite{Lin}); he proves in particular that if $A$ is a $k\times k$ integer matrix 
with nonzero determinant, then there exist two constants $\alpha\geq\beta\geq 0$
 and a unique integer $0\leq\ell\leq k-1$ such that for any~$n\in\mathbb{N}$
\[
\beta\,\rho(A)^\ell\, n^\ell\leq \deg\phi_A^n\leq \alpha\,\rho(A)^\ell\, n^\ell
\]
where $\rho(A)$ denotes the spectral radius of $A$ and $\phi_A$ the monomial map 
associated to $A$. Does there exist in dimension $k$ birational maps of 
$\mathbb{P}^k_\mathbb{C}$ with growth $n^\ell$ with $\ell>k-1$ ? 

\begin{theoalph}\label{thm:newdegreebir}
{\sl Assume $k\geq 3$; forall $0\leq\ell\leq k$ 
there exist birational maps $\phi$ of $\mathbb{P}^k_\mathbb{C}$ such that 
$\deg \phi^n\sim n^\ell$.}
\end{theoalph}

Note that there exists a birational self map $f$ of $\mathbb{P}^3_\mathbb{C}$ such that $\deg f^n\sim n^4$ (\emph{see} \cite{Urech}).

\bigskip

\subsection*{Acknowledgement}   
I am very grateful to Dominique Cerveau for his constant support. 
\tableofcontents

\section{Recalls, definitions, notations}

\subsection{The group of polynomial automorphisms of $\mathbb{C}^k$}\label{subsec:defaut}

A \textbf{\textit{polynomial automorphism}} $f$ of $\mathbb{C}^k$ is a polynomial 
map of the type 
\[
f\colon\mathbb{C}^k\to\mathbb{C}^k,\quad\quad\big(z_0,z_1,\ldots,z_{k-1})\mapsto
(f_0(z_0,z_1,\ldots,z_{k-1}),f_1(z_0,z_1,\ldots,z_{k-1}),\ldots,f_{k-1}(z_0,z_1,\ldots,z_{k-1})
\big)
\]
that is bijective. The set of polynomial automorphisms of $\mathbb{C}^k$ form a 
group denoted $\mathrm{Aut}(\mathbb{C}^k)$. 

The automorphisms of $\mathbb{C}^k$ of the form $(f_0,f_1,\ldots,f_{k-1})$ where 
$f_i$ depends only on $z_i$, $z_{i+1}$, $\ldots$, $z_{k-1}$ are called 
\textbf{\textit{elementary automorphisms}} and form a subgroup $\mathrm{E}_k$ of 
$\mathrm{Aut}(\mathbb{C}^k)$. Moreover we have the inclusions
\[
\mathrm{GL}(\mathbb{C}^k)\subset\mathrm{Aff}_k\subset\mathrm{Aut}(\mathbb{C}^k)
\]
where $\mathrm{Aff}_k$ denotes the \textbf{\textit{group of affine maps}}
\[
f\colon(z_0,z_1,\ldots,z_{k-1})\mapsto\big(f_0(z_0,z_1,\ldots,z_{k-1}),f_1(z_0,z_1,
\ldots,z_{k-1}),\ldots,f_{k-1}(z_0,z_1,\ldots,z_{k-1})\big)
\]
with $f_i$ affine; $\mathrm{Aff}_k$ is the semi-direct product of 
$\mathrm{GL}(\mathbb{C}^k)$ with the commutative subgroups of translations. The 
subgroup $\mathrm{Tame}_k\subset\mathrm{Aut}(\mathbb{C}^k)$ generated by 
$\mathrm{E}_k$ and $\mathrm{Aff}_k$ is called the \textbf{\textit{group of tame 
automorphisms}}. If $k=2$ one has:

\begin{thm}[\cite{Jung}]
{\sl In dimension $2$ the group of tame automorphisms coincides with the whole 
group of polynomial automorphism; more precisely 
\[
\mathrm{Aut}(\mathbb{C}^2)=\mathrm{Aff}_2\ast_{\mathrm{Aff}_2\cap\mathrm{E}_2} \mathrm{E}_2.
\]}
\end{thm}
\noindent But this is not the case in higher dimension: 
$\mathrm{Tame}_3\subsetneq\mathrm{Aut}(\mathbb{C}^3)$ (\emph{see} 
\cite{ShestakivUmirbaev}).

Another important result in dimension $2$ is the following:

\begin{thm}[\cite{FriedlandMilnor}]
{\sl Let $f$ be an element of $\mathrm{Aut}(\mathbb{C}^2)$. Then, up to conjugacy,
\begin{itemize}
\item either $f$ belongs to $\mathrm{E}_2$,

\item or $f$ can be written as 
\[
\varphi_\ell\circ\varphi_{\ell-1}\circ\ldots\varphi_1
\]
where $\varphi_i\colon(z_0,z_1)\mapsto (z_1,P_i(z_1)-\delta_iz_0)$, 
$\delta_i\in\mathbb{C}^*$, $P_i\in\mathbb{C}[z_1]$, $\deg P_i\geq 2$.
\end{itemize}}
\end{thm}

We denote by $\mathcal{H}$ the set of polynomial automorphisms of $\mathbb{C}^2$ 
can be written up to conjugacy as 
\[
\varphi_\ell\circ\varphi_{\ell-1}\circ\ldots\varphi_1
\]
where $\varphi_i\colon(z_0,z_1)\mapsto (z_1,P_i(z_1)-\delta_iz_0)$, 
$\delta_i\in\mathbb{C}^*$, $P_i\in\mathbb{C}[z_1]$, $\deg P_i\geq 2$.
The elements of $\mathcal{H}$ are \textbf{\textit{of H\'enon type}}.

From now one we will denote $f=(f_0,f_1,\ldots,f_{k-1})$ instead of 
\[
f\colon(z_0,z_1,\ldots,z_{k-1})\mapsto\big(f_0(z_0,z_1,\ldots,z_{k-1}),f_1(z_0,z_1,
\ldots,z_{k-1}),\ldots,f_{k-1}(z_0,z_1,\ldots,z_{k-1})\big).
\]

The \textbf{\textit{algebraic degree}} $\deg f$ of 
$f=(f_0,f_1,\ldots,f_{k-1})\in\mathrm{Aut}(\mathbb{C}^k)$ is 
$\max(\deg f_0,\deg f_1,\ldots,\deg f_{k-1})$. 

\subsection{The Cremona group}\label{subsec:cremona}

A \textbf{\textit{rational self map}} $f$ of $\mathbb{P}^k_\mathbb{C}$ can be 
written
\[
\big(z_0:z_1:\ldots:z_k\big)\dashrightarrow\big(f_0(z_0,z_1,\ldots,z_k):
f_1(z_0,z_1,\ldots,z_k):\ldots:f_k(z_0,z_1,\ldots,z_k)\big)
\]
where the $f_i$'s are homogeneous polynomials of the same degree $\geq 1$ and 
without common factor of positive degree. The \textbf{\textit{degree}} of $f$ 
is the degree of the $f_i$. If there exists a rational self map 
$g$ of $\mathbb{P}^k_\mathbb{C}$ such that $fg=gf=\mathrm{id}$ we 
say that the rational self map $f$ of $\mathbb{P}^k_\mathbb{C}$ is 
\textbf{\textit{birational}}. The set of birational self maps of 
$\mathbb{P}^k_\mathbb{C}$ form a group denoted $\mathrm{Bir}(\mathbb{P}^k_\mathbb{C})$ 
and called the \textbf{\textit{Cremona group}}. Of course 
$\mathrm{Aut}(\mathbb{C}^k)$ is a subgroup of $\mathrm{Bir}(\mathbb{P}^k_\mathbb{C})$. 
An other natural subgroup of $\mathrm{Bir}(\mathbb{P}^k_\mathbb{C})$ is the group 
$\mathrm{Aut}(\mathbb{P}^k_\mathbb{C})\simeq\mathrm{PGL}(k+1;\mathbb{C})$ of 
automorphisms of $\mathbb{P}^k_\mathbb{C}$.

The \textbf{\textit{indeterminacy set}} $\mathrm{Ind}(f)$ of $f$ is the set 
of the common zeros of the $f_i$'s. The \textbf{\textit{exceptional set}} 
$\mathrm{Exc}(f)$ of~$f$ is the (finite) union of subvarieties $M_i$ of 
$\mathbb{P}^k_\mathbb{C}$ such that $f$ is not injective on any open subset of 
$M_i$.

\subsection{A little bit of dynamics}\label{subsection:defdyn}

Let $f$ be a polynomial automorphism of $\mathbb{C}^k$. One can see $f$ as a 
birational self map of $\mathbb{P}^k_\mathbb{C}$ still denoted $f$. We will say 
that $f$ is \textbf{\textit{algebraically stable}} if for any $n>0$
\[
f^n\big(\{z_k=0\}\smallsetminus\mathrm{Ind}(f^n)\big)
\]
is not contained in $\mathrm{Ind}(f)$.  This is equivalent to the fact that 
$(\deg f)^n=\deg f^n$ for any $n>0$. For instance elements of $\mathcal{H}$ 
are algebraically stable.

\begin{rem}
Note that in dimension $2$ one usually says that $f$ is algebraically stable 
if for any $n>0$
\[
f^n\big(\{z_2=0\}\smallsetminus\mathrm{Ind}(f^n)\big)\cap\mathrm{Ind}(f)=\emptyset.
\]
Be careful this is not equivalent in higher dimension: consider for instance 
\[
f=\big(5z_0^2+z_2^2+6z_0z_2+z_1,z_2^2+z_0,z_2\big);
\]
then
\begin{itemize}
\item on the one hand $(-1:0:1:0)$ belongs to 
$\{z_3=0\}\smallsetminus\mathrm{Ind}(f)$ and 
$f(-1:0:1:0)=(0:1:0:0)\in\mathrm{Ind}(f)=\{(0:1:0:0)\}$.
\item on the other hand for any $n\geq 1$ $\deg f^n=(\deg f)^n$.  
\end{itemize}
\end{rem}

The algebraic degree of a birational map $f$ of $\mathbb{P}^n_\mathbb{C}$ (resp. 
a polynomial automorphism of $\mathbb{C}^k$) is not a dynamical invariant so we 
introduce the \textbf{\textit{dynamical degree}}
\[
\lambda(f)=\lim_{n\to +\infty}(\deg f^n)^{1/n}
\]
which is a dynamical invariant. In other words if 
$f\in\mathrm{Bir}(\mathbb{P}^k_\mathbb{C})$ (resp. $\mathrm{Aut}(\mathbb{C}^k)$), 
then for any $g\in\mathrm{Bir}(\mathbb{P}^k_\mathbb{C})$ (resp. 
$g\in\mathrm{Aut}(\mathbb{C}^k)$) one has $\lambda(f)=\lambda(gfg^{-1})$. For 
any element $f$ in $\mathcal{H}$ the algebraic and dynamical degrees coincide; 
more precisely if 
\[
f=\varphi_\ell\circ\varphi_{\ell-1}\circ\ldots\circ\varphi_1
\]
where $\varphi_i=(z_1,P_i(z_1)-\delta_iz_0)$, $\delta_i\in\mathbb{C}^*$, 
$P_i\in\mathbb{C}[z_1]$, $\deg P_i\geq 2$ one has (\cite{FriedlandMilnor})
\[
\lambda(f)=\prod_{i=1}^\ell \deg \varphi_i\geq 2.
\]
A polynomial automorphism $f$ is in $\mathrm{E}_k$ if and only if for any 
$n\geq 1$ the equality $\deg f=\deg f^n$ holds hence $\lambda(f)=1$. In other 
words a polynomial automorphism $f$ of $\mathbb{C}^2$ belongs to $\mathcal{H}$
 if and only if $\lambda(f)>1$. There is an other characterization of the 
automorphisms of H\'enon type: 

\begin{thm}[\cite{Lamy}]\label{thm:steph}
{\sl The centralizer of $f$ in $\mathrm{Aut}(\mathbb{C}^2)$, that is 
$\big\{g\in\mathrm{Aut}(\mathbb{C}^2)\,\vert\, fg=gf\big\}$, is countable if 
and only if $f$ belongs to $\mathcal{H}$.}
\end{thm}

\section{Automorphisms with polynomial growths}\label{sec:lineargrowths}

\subsection{The growths of a polynomial automorphisms and its inverse}

If $f$ is a polynomial automorphism of~$\mathbb{C}^k$, then 
$(\deg f,\deg f^{-1})$ is the bidegree of $f$. There is a relationship
between $\deg f$ and $\deg f^{-1}$ (\emph{see}~\cite{Bass}):
\begin{equation}\label{eq:reldeg}
\left\{
\begin{array}{ll}
\deg f^{-1}\leq(\deg f)^{k-1}\\
\deg f\leq(\deg f^{-1})^{k-1}
\end{array}
\right.
\end{equation}

As a result if $f$ is a polynomial automorphism of $\mathbb{C}^k$, the degree 
growths of $f$ and $f^{-1}$ are linked:
 
\begin{pro}
{\sl Let $f$ be a polynomial automorphism of $\mathbb{C}^k$.

\begin{itemize}
\item The sequence $(\deg f^n)_{n\in\mathbb{N}}$ is bounded if and only if the sequence 
$(\deg f^{-n})_{n\in\mathbb{N}}$ is bounded.
\smallskip
\item The sequence $(\deg f^n)_{n\in\mathbb{N}}$ grows exponentially if and only if 
$(\deg f^{-n})_{n\in\mathbb{N}}$ grows exponentially.
\smallskip
\item If $\deg f^n\simeq n^p$ and $\deg f^{-n}\simeq n^q$ for some integers $p,$ 
$q\geq 1$, then
\[
(p,q)\in\left\{\left(\Big[\frac{q+1}{k}\Big],q\right),\ldots,(kq,q)\right\}.
\]
\end{itemize}}
\end{pro}

\begin{rem}
When we write "The sequence $(\deg f^n)_{n\in\mathbb{N}}$ grows exponentially if 
and only if $(\deg f^{-n})_{n\in\mathbb{N}}$ grows exponentially" it does not mean 
that $(\deg f^n)_{n\in\mathbb{N}}$ and $(\deg f^{-n})_{n\in\mathbb{N}}$ have exactly 
the same behavior: the polynomial automorphism of $\mathbb{C}^3$ given by 
$f=(z_0^2+z_1+z_2,z_0^2+z_1,z_0)$ satisfies forall $n\geq 1$
\[
\left\{
\begin{array}{lll}
\deg f^n=2^n\\
\deg f^{-n}=2^{[\frac{n+1}{2}]}
\end{array}
\right.
\]
\end{rem}

\subsection{Examples of polynomial automorphisms with new polynomial growths}

Let us now give examples of polynomial automorphisms with polynomial growths.

\begin{lem}\label{tec1}
{\sl Let us consider the polynomial automorphism of $\mathbb{C}^3$ given by
\[
f=\big(z_1+z_0z_2^d,z_0,z_2\big)
\]
where $d\geq 1$.
One has $\deg f^n=dn+1$ and $\deg f^n=\deg f^{-n}$ for any $n\geq 1$.}
\end{lem}

\begin{proof}
Assume $n\geq 1$. Set $f^n=(f_{0,n},f_{1,n},z_2)$ and $\delta_n=\deg f^n$. Note 
that $\delta_1=d+1$ and since 
\[
f^n=ff^{n-1}=(f_{1,n-1}+f_{0,n-1}z_2^d,f_{0,n-1},z_2)
\]
one has 
\[
\delta_n=\max\big(\deg f_{1,n-1},\deg f_{0,n-1}+d,\deg f_{0,n-1},1\big).
\]
But $f_{1,n-1}=f_{0,n-2}$ and $d\geq 1$ so $\delta_n=\deg f_{0,n-1}+d=\delta_{n-1}+d$.
\end{proof}

Similarly one can prove:

\begin{lem}\label{tec4}
{\sl Let us consider the polynomial automorphism of $\mathbb{C}^5$ given by
\[
g=\big(z_1+z_0z_2^d,z_0,z_2,z_4+z_0^pz_3,z_3\big)
\]
where $p\geq d\geq 1$. One has for any $n\geq 1$
\[
\deg g^n=\frac{pd}{2}\,n^2+\frac{p(2-d)}{2}\,n+1
\]
and $\deg g^n=\deg g^{-n}$ for any $n\geq 1$.

Let us consider the polynomial automorphism of $\mathbb{C}^7$ given by
\[
h=\big(z_1+z_0z_2^d,z_0,z_2,z_4+z_0^pz_3,z_3,z_6+z_3^\ell z_5,z_5\big)
\]
where $\ell\geq p\geq d\geq 1$.
For any $n\geq 1$
\[
\deg h^n=1+\ell\left(1-\frac{p}{2}+\frac{pd}{3}\right)n+\frac{\ell p(1-d)}{2}
\,n^2+\frac{\ell pd}{6}\,n^3
\]
and $\deg h^n=\deg h^{-n}$.
}
\end{lem}

So Lemma \ref{tec1} gives an example $f$ of polynomial automorphism of $\mathbb{C}^3$ 
with linear growth; from $f$ one gets a polynomial automorphism $g$ of $\mathbb{C}^5$ 
such that $\deg g^n\sim n^2$, and from $g$ one gets a polynomial automorphism $h$ of 
$\mathbb{C}^7$ such that $\deg h^n\sim n^3$ (Lemma \ref{tec4}). By repeating this 
process one gets the following statement:

\begin{pro}\label{pro:ex}
{\sl There exist polynomial automorphisms $f$ of $\mathbb{C}^{2k+1}$, $k\geq 2$, such 
that 
\[
\deg f^n\sim n^k.
\]
}
\end{pro}

Theorem \ref{thm:newdegree} follows from Proposition \ref{pro:ex}.

\subsection{A consequence}

Come back to $f=(z_1+z_0z_2^d,z_0,z_2)$ and consider the birational map of 
$\mathbb{P}^4_\mathbb{C}$ given by $F=(f,z_0^pz_3)$, that is $F=(z_1+z_0z_2^d,z_0,z_2,z_0^pz_3)$. 
Assume that $d\leq p$. Then one can prove by induction that for any $n\geq 1$
\[
\deg F^n=\frac{pd}{2}\,n^2+\frac{p(2-d)}{2}\,n+1.
\]
Let us now define the birational map $F$ of $\mathbb{P}^5_\mathbb{C}$ by
\[
G=(z_1+z_0z_2^d,z_0,z_2,z_0^pz_3,z_3^\ell z_4);
\]
for any $n\geq 1$ one has
\[
\deg G^n=\frac{\ell pd}{6}\,n^3+\, \left(1-\frac{3d}{4}\right)\ell p n^2+\left(\frac{13}{12}pd-2p+1\right)\,\ell n-\frac{\ell pd}{2}+\ell p+1.
\]

Repeating this process one gets 

\begin{pro}
{\sl If $k\geq 3$, then for any $0\leq \ell\leq k-2$ 
there exist birational maps $F$ of $\mathbb{P}^k_\mathbb{C}$ such that 
\[
\deg F^n\sim n^\ell.
\]}
\end{pro}

\section{Properties $(\mathcal{P}_i)$}\label{sec:countereg}

Note that Theorem \ref{thm:steph} can be also stated as follows: the centra\-lizer of a 
polyonomial automorphism $f$ of $\mathbb{C}^2$ is countable if and only if 
$(\deg f^n)_{n\in\mathbb{N}}$ grows exponentially. This property is not true in higher 
dimension: there exist polynomial automorphisms of $\mathbb{C}^3$ with uncountable 
centralizer and exponential degree growth (\cite{Bisi}). Let us now that this is also 
the case for other properties, and in particular for $(\mathcal{P}_1)$, 
$(\mathcal{P}_2)$, $(\mathcal{P}_3)$, $(\mathcal{P}_4)$.

\subsection{Property $(\mathcal{P}_1)$}

Property $(\mathcal{P}_1)$ is not satisfied in dimension $3$:

\begin{pro}
{\sl If 
\[
f=\big(z_2,(z_2^2+z_0)^2+z_2^2+z_0+z_1,z_2^2+z_0\big)
\]
then 
\begin{itemize}
\item $\mathrm{Ind}(f)=\{z_2=z_3=0\}$ and $z_3=0$ is blown down 
by $f$ onto $(0:1:0:0)\in\mathrm{Ind}(f)$,

\item for all $n\geq 1$ one has $\deg f^n=2^{n+1}$.
\end{itemize}}
\end{pro}

\subsection{Property $(\mathcal{P}_2)$}

Property $(\mathcal{P}_2)$ does not hold in higher dimension:

\begin{pro}
{\sl The polynomial automorphism $f$ of $\mathbb{C}^3$ given by 
\[
\big(z_0^2+z_1,z_0,z_2+1\big)
\]
preserves the fibration $z_2=$ cst and for all $n\geq 1$ the equality 
$\deg f^n=2^n$ holds.

The polynomial automorphism $g$ of $\mathbb{C}^3$ given by 
\[
\big(z_1^2+z_0z_1+z_2,z_1+1,z_0\big)
\]
preserves the fibration $z_1=$ cst and for all $n\geq 1$ the equality 
$\deg g^n=n+1$ holds.}
\end{pro}

\subsection{Property $(\mathcal{P}_3)$}

Property $(\mathcal{P}_3)$ is not satisfied in dimension $3$.

\begin{pro}
{\sl Let us consider the polynomial automorphisms $f$ and $g$ of $\mathbb{C}^3$ 
defined by 
\begin{align*}
&f=\big(z_0+z_1+z_2,z_0^2+z_0+z_1,z_0\big),
&&
g=\big(z_1^2+z_0+z_1+z_2,z_1,z_0\big).
\end{align*}
For all $n\geq 1$ one has 
\[
\left\{
\begin{array}{lll}
\deg f^{2n}=2^{n+1},\,\deg f^{2n+1}=2^{n+1} \\
\deg f^{-n}=2^n\\
\deg g^n=\deg g^{-n}=2
\end{array}
\right.
\]

The automorphism $f$ sends $z_3=0$ onto $(0:1:0:0)$ and $f^{-1}$ sends $z_3=0$ onto 
$(0:1:1:0)$. Furthermore $\mathrm{Ind}(f)=\{z_0=z_3=0\}$ and  
$\mathrm{Ind}(f^{-1})=\{z_2=z_3=0\}$.

The automorphism $g$ sends $z_3=0$ onto $(1:0:0:0)$ and $g^{-1}$ sends $z_3=0$ onto 
$(0:0:1:0)$. Besides  $\mathrm{Ind}(g)=\{z_1=z_3=0\}$ and
$\mathrm{Ind}(g^{-1})=\{z_2=z_3=0\}$.
}
\end{pro}

\subsection{Property $(\mathcal{P}_4)$}

In \cite{Furter} Furter proves that if $f$ is a polynomial automorphism of $\mathbb{C}^2$ then
$\deg f^2=(\deg f)^2$ if and only if $\deg f^n=(\deg f)^n$ for all $n\in\mathbb{N}$. 
This property does not hold in higher dimension: consider for 
instance the polynomial automorphism $f$ given by 
\[
f=(z_1^2+z_5,z_5^2+z_4,z_2,z_1,z_0,z_4^2+z_3).
\]
One can check that $\deg f=2$, $\deg f^2=4$, $\deg f^3=8$ but $\deg f^4=8$.

\medskip

Let $f$ be a polynomial automorphism of 
$\mathbb{C}^k$. For any integer $n\geq 0$ set 
$\Omega_n=f^n\big((z_{k-1}=0)\smallsetminus\mathrm{Ind}(f^n)\big)$; note that 
$\Omega_n\subset(z_{k-1}=0)$ for any $n$. We say that 
\textbf{\textit{$f$ is not algebraically stable after $\ell$ steps}} if 
$\ell$ is the smallest integer such that $\Omega_\ell\subset\mathrm{Ind}(f)$.

\smallskip 

Let us first remark that if $\Omega_1\cap\mathrm{Ind}(f)=\emptyset$, then
$\Omega_2=f(\Omega_1)\subseteq\Omega_1$ so 
$\Omega_2\cap\mathrm{Ind}(f)=\emptyset$. By induction one gets for any 
$n\geq 1$ that 
$\Omega_n\cap\mathrm{Ind}(f)=\emptyset$ and $\deg f^n=(\deg f)^n$, \emph{i.e.}
$f$ is algebraically stable. 

\smallskip

Let us now assume that $\Omega_1\cap\mathrm{Ind}(f)\not=\emptyset$. Then:

\begin{enumerate}
\item Either $\Omega_1\subset\mathrm{Ind}(f)$, that is $f$ is not algebraically
stable after $1$ step.

\item Or $\Omega_1\not\subset\mathrm{Ind}(f)$ hence
$\Omega_2=f\big(\Omega_1\smallsetminus\mathrm{Ind}(f)\big)\subseteq\Omega_1$.
\begin{itemize}
\item Either $\dim\Omega_2=\dim\Omega_1$, so $\Omega_2=\Omega_1$ and then 
$\Omega_n=\Omega_1$ for any $n$; in particular 
$\Omega_n\not\subset\mathrm{Ind}(f)$ for any $n$ and $f$ is algebraically 
stable.

\item Or $\dim\Omega_2<\dim\Omega_1$, then either 
$\Omega_2\subset\mathrm{Ind}(f)$ and $f$ is not algebraically stable after
$2$ steps or $\Omega_2\not\subset\mathrm{Ind}(f)$ and we come back to the 
previous alternative, that is either 
$\dim\Omega_3=\dim\Omega_2$ or $\dim\Omega_3<\dim\Omega_2$. Since for any 
$n$ one has $0\leq\dim\Omega_n\leq k-1$ one gets that either $f$ is 
algebraically stable, or $f$ is not algebraically stable after at most
$k-1$ steps.
\end{itemize}
\end{enumerate}

Hence one can state:

\begin{pro}
{\sl Let $f$ be a polynomial automorphism of $\mathbb{C}^k$. Either
$f$ is algebraically stable, or $f$ is not algebraically stable after
$\ell$ steps, with $\ell\leq k-1$.

In other words $\deg f^i=(\deg f)^i$ for $1\leq i\leq k$ if and only if $\deg f^n=(\deg f)^n$ for any $n\geq 1$.}
\end{pro}

\section{Birational maps with new polynomial growths}

Let us first recall that 

\begin{lem}[\cite{DillerFavre}]
{\sl The birational self map $\varphi$ of $\mathbb{P}^2_\mathbb{C}$ given in the affine chart $z_2=1$ by
\[
\varphi(z_0,z_1)=\left(z_1+\frac{2}{3},z_0\frac{z_1-\frac{1}{3}}{z_1+1}\right)
\]
satisfies $\deg \varphi^n\sim n^2$.}
\end{lem}

Denote $\varphi^n$ by $\left(\frac{P_n}{Q_n},\frac{R_n}{S_n}\right)$ where $P_n$, $Q_n$, $R_n$ and $S_n$ denote some 
elements of $\mathbb{C}[z_0,z_1]$ without common factor. 
Set $p_n=\deg P_n$, $q_n=\deg Q_n$, $r_n=\deg R_n$
and $s_n=\deg S_n$. The following equalities hold (by iteration)
\[
\left\{
\begin{array}{llll}
p_n=s_{n-1}+1\\
q_n=s_{n-1}\\
r_n=s_n+1\\
\deg \varphi^n=s_{n-1}+s_n+1
\end{array}
\right.
\]

Let us now consider the birational self map of $\mathbb{P}^3_\mathbb{C}$ given in the affine chart $z_3=1$ by
\[
\Psi_3(z_0,z_1,z_2)=\left(z_1+\frac{2}{3},z_0\frac{z_1-\frac{1}{3}}{z_1+1},z_0z_2\right).
\]
One can also check that 
\[
\Psi_3=\left(\frac{P_n}{Q_n},\frac{R_n}{S_n},\frac{U_n}{V_n}\right) 
\]
where $U_n=z_0z_2P_1P_2\ldots P_{n-1}$ and $V_n=Q_1Q_2\ldots Q_{n-1}$ have no common factor. Since $s_i\sim i^2$ one gets:

\begin{lem}
{\sl The birational self map $\Psi_3$ of $\mathbb{P}^3_\mathbb{C}$ satisfies
\[
\deg \Psi_3\sim n^3.
\]}
\end{lem}

Let us now consider the birational self map of $\mathbb{P}^4_\mathbb{C}$ defined in the affine chart $z_4=1$ by 
\[
\Psi_4(z_0,z_1,z_2,z_3)=\left(z_1+\frac{2}{3},z_0\frac{z_1-\frac{1}{3}}{z_1+1},z_0z_2,z_2z_3\right).
\]

One can check that $\Psi_4=\left(\frac{P_n}{Q_n},\frac{R_n}{S_n},\frac{U_n}{V_n},\frac{W_n}{X_n}\right)$ 
where $W_n=W_1U_1U_2\ldots U_{n-1}$ and $X_n=X_1V_1V_2\ldots V_{n-1}$ have no common factor. Since $\deg U_n\sim n^3$ and $\deg V_n\sim n^3$ one has:

\begin{lem}
{\sl The birational self map $\Psi_4$ of $\mathbb{P}^4_\mathbb{C}$ satisfies
\[
\deg \Psi_4\sim n^4.
\]}
\end{lem}

By repeating this process one gets:

\begin{thm}
{\sl Let $k$ be an integer $\geq 3$. There exist birational maps $\phi$ of 
$\mathbb{P}^k_\mathbb{C}$ such that 
\[
\deg \phi^n\sim n^k.
\]}
\end{thm}

This statement and Lin's result (\cite{Lin}) imply Theorem \ref{thm:newdegreebir}.

\vspace{8mm}

\bibliographystyle{plain}
\bibliography{biblio}
\nocite{}

\end{document}